\documentclass{amsart}

\usepackage{graphics, graphicx}
\usepackage{slashbox}
\usepackage[justification=centering]{caption}
\usepackage[ruled]{algorithm2e} 
\usepackage{multicol}
\usepackage{multirow}

\newcommand{\dsum}{\displaystyle\sum}
\newcommand{\dmax}{\displaystyle\max}

\usepackage{subcaption}

\newtheorem{theo}{Theorem}

\newtheorem{defn}[theo]{Definition}
\newtheorem{ex}[theo]{Example}

\newtheorem{prop}[theo]{Proposition}
\newtheorem{remark}[theo]{Remark}
\def\R{\mathbb{R}}
\def\RP{\mathrm{RP}}
\def\CP{\mathrm{CP}}
\def\BP{\mathrm{BP}}
\def\OWA{\mathrm{OWA}}

\def\Poi{\mathcal{P}oisson}
\def\Geo{\mathrm{Geometric}}
\def\Beta{\mathrm{Beta}}
\def\Ga{\mathrm{Gamma}}
\def\bmp{\mathrm{BPM}}
\newcommand{\dmin}{\displaystyle\min}

\usepackage{pgfplots}
\pgfplotsset{compat=newest}
\usetikzlibrary{decorations.markings}    

\usepackage{hyperref}

\title{On the aggregation of experts' information in Bonus--Malus systems}

\author{V\'ictor Blanco}
\address{Dept. Quantitative Methods for Economics \& Business\\
Universidad de Granada.}
\email{vblanco@ugr.es}
\author{Jos\'e M. P\'erez--S\'anchez}
\address{Dept. Quantitative Methods for Economics \& Business\\
Universidad de Granada.}
\email{josemag@ugr.es}


\begin{document}

\maketitle

\begin{abstract}
We present in this paper a new premium computation principle based on the use of \textit{prior} information from multiple sources for computing the premium charged to a policyholder. Under this framework, based on the use of Ordered Weighted Averaging (OWA) operators, we propose alternative collective and Bayes premiums and describe some approaches to compute them. Several examples illustrates the new framework for premium computation.
\end{abstract}

\section{Introduction}

A Bonus--Malus System (BMS) is a merit rating method that is widely used by actuaries in non--life insurance, specially for  automobile insurances (see \cite{gomez02}, \cite{gomez06}, \cite{lemaire79}, \cite{lemaire88}, \cite{lemaire98}, \cite{shengwang-whitmore99} for a brief overview of this topic). This rating method is based on the principle that the premium charged to a client or a portfolio of clients depends on the claims filed on the policy in the past. The BMS methodology is characterized by important penalties in case of claims and moderate premium discounts awarded to
claim--free policyholders. Due to the nature of the problem it is common to apply Bayesian techniques to compute those premiums. In the Bayesian paradigm, a \textit{prior} distribution is used to quantify the knowledge about the unknown (risk) parameter,  while in case data are available, the \textit{prior} knowledge is updated using the likelihood function to give the \textit{posterior} distribution. One of the most usual techniques to compute BM premiums consists of minimizing certain expected losses that measure the deviation of the estimated risk parameter with respect to its actual value. This is performed by plugging the risk into the distribution function of the number of claims or severity (see \cite{lemaire79,mert2005,shengwang-whitmore99} among others). The parameters of the \textit{prior} distributions are usually provided by the expert's knowledge, while the \textit{posterior} distribution also considers a sample of claims from the policyholders during a previous time horizon.

Bayesian methodology assumes that prior information is provided by an expert, who should manifest some (sometimes partial) information about the behavior of the portfolio of policyholders in terms of claims. If this ``experience'' is not fully available, the \textit{prior knowledge} can be estimated from a given sample. This information is not always unified and different experts may manifest different behaviors of the same risk. Furthermore, the information provided by the experts is not always precisely reported since they are usually asked to provide some approximate descriptive measures of the risk as the mean or the mode. This imprecision might imply a considerable monetary gain or loss to the insurance company. In this paper, we deal with the problem of computing unified Bonus--Malus premiums when different experts give their opinion about the risk of a portfolio of policyholders.

The problem of aggregating different experts' information is not new and some authors have previously proposed techniques to perform such an aggregation by assuming that each expert provides a probabilistic information over the risks under study. Mixtures of distributions allows to provide a single distribution aggregating the whole information (the reader is referred to \cite{clemen-winkler}, \cite{genest-zidek} and \cite{ohagan} for further information). These techniques are usually called ``opinion pools''. The most common distributions under this framework are the linear, in which the probability distributions of the experts are averaged, and logarithmic (see \cite{bacharach}, \cite{rufo-perez} and \cite{stone}). Also, in \cite{arbenz-canestraro} a Bayesian method is presented that combines different sources of information to estimate copula parameters in a scarce observations context.

In this paper, we propose a new approach based on aggregating the \textit{prior} knowledge of all the experts on an unique loss function which is minimized to compute the premium by considering only the distribution of the number of claims. Such an aggregation criterion  generalizes the linear opinion pools methodology and is based on the family of Ordered Weighted Averaging (OWA) operators. These operators were introduced in \cite{yager88}, as a flexible approach to aggregate functions in multicriteria decision making. They allow to model as particular cases the maximum, the mean, the minimum, the mean or the range of a finite set of values. In OWA aggregation, a vector of weights is given. The result is the linear combination of those weights with the ordered values, the first weight to the largest value, the second weight to the second largest value, etc. OWA operators have been applied in many different fields in order to generalize classical operators. They have been used to allow flexibility in many models and that combined with mathematical programming tools have been applied to generalize Linear Regression models \cite{BPS15}, the Gini index \cite{cardin2014}, Location problems \cite{nickel-puerto}, among others. 

The main highlights of our approach are: (1) The premium assigned to a client depends not only of a single \textit{prior} distribution but of several ones (one for each of the experts that give their opinion on the behavior of the portfolio) and does not need previous probabilistic information to aggregate the expert's information; (2) a flexible family of choices for aggregating the \textit{prior} information is provided to the actuary; and (3) an unified mathematical optimization framework is presented to compute the premiums to be loaded to the policyholders. Although not widely extended as it should, mathematical programming tools have been already applied in Risk Theory (see e.g.,  \cite{carretero2000} or \cite{gupta2007}).

The paper is organized as follows. In Section \ref{sec:1} we recall some of the definitions and results in net premium computation principles and OWA operators. Section \ref{sec:2} is devoted to introduce and state the main results of the new framework for computing risk premiums based on the use of several sources of \textit{prior} information aggregated with OWA operators. Finally, Section \ref{sec:3} reports the results of  some computational experiments performed in order to compare our approach with the classical methodology.

\section{Preliminaries}
\label{sec:1}

In this section, for the sake of completeness, we introduce the main definitions that will be used through this paper. We recall some of the main premium computation principles and state the notation for the rest of the sections.

Consider a policyholder, drawn randomly from the insurance portfolio, who is observed to have a sequence of independent and equidistributed claims, $X_1$, $X_2$, $\ldots$ ,$X_t$, representing the performances of a single risk in the last $t$ periods.  A premium computation principle is a function $\mathcal{H}$, that maps to any risk a nonnegative real number, the premium that should be charged for the next period. Some properties are desired to be fulfilled by any premium calculation principle (see \cite{gerber79}, \cite{heilmann89}, \cite{hurlimann94} and \cite{young04} for a detailed description).

A first attempt to compute the premium of a policyholder when no previous experience is known is the \emph{risk premium}. For a given random variable $X$ (number of claims) with density $f_\theta(x)$ for a given risk parameter $\theta \in \R$ and a fixed loss function $\mathcal{L} : \R^2 \rightarrow \R$, where $\mathcal{L}(P,x)$ is the loss incurred by a decision maker taking the action $P$ (the premium paid by the policyholder) and facing the outcome $x$ (the actual number of claims). The unknown risk premium is obtained by minimizing the expected loss under the risk distribution:
\begin{equation}
\RP:= \dmin_{P \in \mathbb{P}} L^R(P) = \int_x \mathcal{L}(P, x) f_\theta(x)  dx = E\left[\mathcal{L}(P,x)\right],\tag{${\rm RP}$}\label{eq:rp}
\end{equation}
 where $\mathbb{P} \subseteq \R$. For the sake of simplicity, we assume that $\mathbb{P} = \R$, although the results in this paper extend to the case when $\mathbb{P}$ is a given interval. Also, we consider quadratic losses in the form:
 $$
 \mathcal{L}(P,x) = (P-x)^2.
 $$

The premiums computed under this loss function are called \emph{net premiums} and the risk loss function becomes:
 $$
 L^R(P) = P^2 - 2P E[X] + E[X^2],
 $$
 being its minimum attained at $P^*_R=E[X]$ with $L^R(P^*_R)=Var(X)$.

In the Bayesian paradigm, the parameters of the risk distribution are assumed to be unknown, but distributed according to a \textit{prior} distribution. In this setting, the \emph{collective premium} is obtained by minimizing the expected loss (under the \textit{prior} distribution) with respect to the actual risk premium:
\begin{equation}
\CP := \dmin_{P \in \mathbb{P}} L^C(P) =  \int_\theta \mathcal{L}(P, P(\theta)) \pi(\theta)  d\theta, \tag{${\rm CP}$}\label{eq:cp}
\end{equation}
where $P(\theta)$ is the risk premium when the parameter of the risk distribution is $\theta$ and  $\pi$ is the density function of the \textit{prior} distribution.

With the net premium principle (quadratic loss), we get that the collective loss function (to be minimized to obtain the premium) is:
\begin{equation}\label{cpe}
L^C(P) = P^2 - 2P E_\pi[E_\theta[X]] + E_\pi[E_\theta[X]^2],
\end{equation}
where $E_{\theta}[X]$ denotes the expected value of $X$ with the (random) parameter $\theta$. Note that for a given \textit{prior} distribution $\pi$, its minimum is attained at $P_C^* = E_\pi[E_\theta[X]]$ with $L^C(P_C^*) = Var_\pi(E_\theta[X])$.

Finally, combining the \textit{prior} information and the sample information $\mathbf{x}$ (\textit{posterior} knowledge), a third premium, the so--called \emph{Bayes premium} is computed as:
\begin{equation}
\BP(\mathbf{x}):= \dmin_{P \in \mathbb{P}}  L^B(P) =\int_\theta \mathcal{L}(P,P(\theta)) \pi(\theta|\mathbf{x})
d\theta, \tag{${\rm BP}$}\label{eq:bp}
\end{equation}
where $\pi(\theta|\mathbf{x})$ is the density function of the \textit{posterior} distribution obtained by applying Bayes' Theorem.

%
Again, for the quadratic loss, the Bayes loss function is:
\begin{equation}\label{bpe}
L^B(P) =  P^2 - 2P E_{\pi(\theta|x)}[E_\theta[X]] + E_{\pi(\theta|x)}[E_\theta[X]^2],
\end{equation}
whose minimum is $P_B^* = E_{\pi(\theta|x)}[E_\theta[X]]$ with $L^B(P_B^*) = Var_{\pi(\theta|x)}(E_\theta[X])$.

We show in what follows, for the sake of a self-contained paper, the shapes of some well-known (collective and Bayes) loss functions in automobile insurance and that will be used in our experiments.

\begin{prop}
\label{prop:poigamma}

Let  $X$  be a risk and $x_1, \ldots, x_t$  a sample data for a given policyholder:
\begin{enumerate}
\item If $X \sim \Poi(\theta)$ and $\theta \sim \Ga(\alpha,\beta)$, then $P^*_R = \theta$ and the collective and Bayes loss functions are, respectively:
\begin{align*}
L^C (P) &= P^2 - 2\frac{\alpha}{\beta}P + \frac{\alpha}{\beta}\left(\frac{\alpha}{\beta}+1\right),\\
L^B (P) &= P^2 - 2\frac{\alpha + t\overline{x}}{\beta+t} P + \frac{\alpha + t\overline{x}}{\beta+t}\left(\frac{\alpha + t\overline{x}}{\beta+t}+1\right).
\end{align*}
Thus, $P_C^* = \dfrac{\alpha}{\beta}$ and $P_B^*=\dfrac{\alpha + t\overline{x}}{\beta+t}$.
\item If $X \sim \Geo(\theta)$ and $\theta \sim \Beta(\alpha,\beta)$, then $P^*_R = \dfrac{1-'\theta}{\theta}$, while the collective and Bayes loss functions are, respectively:
\begin{align*}
\mathcal{L}^C (P) = P^2 - 2 \frac{\beta}{\alpha-1}P+  \frac{(\beta+1)\beta}{(\alpha-1)(\alpha-2)},\\
\mathcal{L}^B (P) = P^2 - 2 \frac{\beta + t\bar x}{\alpha+t-1}P+  \frac{(\beta+t\bar x+1)(\beta+t\bar x)}{(\alpha+t-1)(\alpha+t-2)}.
\end{align*}
Thus, $P_C^*=\dfrac{\beta}{\alpha-1}$ and $P_B^*=\dfrac{\beta + t\bar x}{\alpha+t-1}$.
\end{enumerate}
\end{prop}

With the above definitions, Lemaire \cite{lemaire79} defined Bonus--Malus premium as:
\begin{equation}
\bmp  = 100 \; \dfrac{P^*_B}{P^*_C}, \tag{${\rm BMP}$}\label{eq:bmp}
\end{equation}
which represents the percentage overcharged of the individual premium with respect to the collective premium.

The collective and Bayes premiums are based on \textit{prior} information provided by experts or previous experience. However, in practice, such an information is not completely objective and different opinions may be needed to conform a more precise premium. Here, we propose a methodology to manage the \textit{prior} information provided by different experts to compute premiums. These premiums will be calculated by minimizing an aggregation of the expected losses for each \textit{prior} information. The aggregation will be performed using Ordered Weighted Averaging (OWA) operators. Let $f_1, \ldots, f_n: \R^d \rightarrow \R$, $n$ real--valued mappings to be aggregated. For each $x \in \R^d$, the vector $\mathbf{f}(x)=(f_1(x), \ldots, f_n(x))$ is sorted in non--increasing order, obtaining $(f_{\sigma(1)}(x), \ldots, f_{\sigma(n)}(x))$, for a permutation $\sigma$ on the indices $\{1, \ldots, n\}$, such that $f_{\sigma(1)}(x)\geq \cdots \geq f_{\sigma(n)}(x)$. Abusing of notation, $f_{\sigma(i)}$ will be denoted by $f_{(i)}$ through this paper). Then, for the given vector of weights $\mathbf{\omega} = (\omega_1, \ldots, \omega_n) \in \R^n$, the OWA aggregation function is defined as:
$$
\OWA(\mathbf{f}; \mathbf{\omega}) := \dsum_{i=1}^n \omega_i \; f_{(i)}(x),
$$
where the $i$--th weight, $\omega_i$ is assigned to the value $f_{(i)}(x)$  which is in the $i$--th position when the vector $\mathbf{f}(x)$ is sorted in non--increasing order. Note that, in case $\omega_i=1$ for all $i$, the OWA operator coincides with the mean operator, while if
$\omega_1=1$ and $\omega_i=0$ for $i\neq 1$, it coincides with the maximum operator. Table \ref{table:owa} shows some of the statistical measures that can be derived by specifying different weights to the OWA functions to aggregate the measures $z_1=f_1(x), \ldots, z_n=f_n(x)$.

\begin{table}[h]
\centering
\begin{tabular}{|c|c|c|}\cline{2-3}
\multicolumn{1}{c|}{}  & $\mathbf{\omega}$ & $\OWA(\mathbf{z}_i; \mathbf{\omega})$\\
  \hline
 Sum                 & $(1,\ldots, 1)$   & $\dsum_{i=1}^{n} z_i$\\\hline
 Maximum                   & $(1,0,\ldots,0)$ & $\dmax_{1\leq i\leq n} \;z_i$\\\hline
  Minimum                   & $(0,0,\ldots,1)$ & $\dmin_{1\leq i\leq n} \;z_i$\\\hline
 Median                   & $(\overbrace{0,\ldots,0}^{\lceil\frac{n}{2}\rceil},1, 0,\ldots,0)$ & ${\rm median} \{z_1, \ldots, z_n\}$\\\hline
$K$--Centrum            & $(\overbrace{1,\ldots,1}^{k},0,\ldots,0)$ &  $\dsum_{i=1}^{k} z_{(i)}$\\\hline
anti $K$--Centrum            & $(0,\ldots,0, \overbrace{1,\ldots,1}^{k})$ & $\dsum_{i=n-k+1}^{n} z_{(i)}$\\\hline
 $(k_1,k_2)$--Trimmed mean& $(\overbrace{0,\ldots,0}^{k_1},1,\ldots,1,\overbrace{0,\ldots,0}^{k_2})$ & $\dsum_{i=k_1+1}^{n-k_2} z_{(i)}$\\\hline
 Range                     & $(1, 0, \cdots, 0 ,-1)$ & $\dmax_{1\leq i\leq n} z_{i}-\dmin_{1\leq i\leq n} z_{i}$\\\hline
\end{tabular}
\caption{Some of the choices for the OWA operators.\label{table:owa}}
\end{table}

\section{OWA Bonus--Malus Systems}
\label{sec:2}

In this section we provide a general framework to deal with the problem of computing premiums when different sources of \textit{prior} information are provided. Let $X$ be a random variable, representing the number of claims for a given portfolio, with density $f_\theta(x)$, depending on an unknown risk parameter $\theta$. We assume that $n$ experts have provided \textit{prior} knowledge about the risk. Hence, $n$ \textit{prior} densities $\pi_1, \ldots, \pi_n$ are given to explain the behavior of $\theta$.

If each single \textit{prior} distribution is taken into account, one may has different collective \eqref{eq:cp} and Bayes \eqref{eq:bp} risk premiums, one for each of the experts. However, different premiums cannot be loaded to a single policyholder. In what follows, we describe how to handle with such an information, by taking into account all the experts' experience.

Let us denote by $\mathbf{c} = (c_1, \ldots, c_n) \in \R_+^n$ each of the confidence level weights for each of the experts.  Let us also denote by $\mathbf{\omega} = (\omega_1, \ldots, \omega_n) \in \R^n$ the OWA weights. We assume, without loss of generality, that $\dsum_{i=1}^n c_i=1$ and that $\omega_i \leq 1$ for $i=1, \ldots, n$. For each $i \in \{1, \ldots, n\}$, we denote by $L_i^C$ (resp. $L^B_i$) the collective (resp. Bayes) loss function for the $i$th expert, i.e.,
\begin{align*}
L_i^C (P)  &=  P^2 - 2P E_{\pi_i}[E_\theta[X]] + E_{\pi_i}[E_\theta[X]^2],\\
L_i^B (P)  &=  P^2 - 2P E_{\pi_i(\theta|x)}[E_\theta[X]] + E_{\pi_i(\theta|x)}[E_\theta[X]^2].
\end{align*}

Observe that the loss functions model the behaviour of the losses according to the $i$th prior (collective) or posterior (Bayes) experts' information.

In these setting we introduce our flexible alternatives to collective and Bayes risk premiums as follows:

\begin{defn}[Collective and Bayes OWA Premiums]
$ $
\begin{itemize}
\item The OWA collective Premium is defined as:
\begin{equation}\label{eq:owacol}
\CP_{\mathbf{\omega}} \in  \arg\dmin_{P \in \mathbb{P}} \OWA( (c_1L^C_1(P), \ldots, c_nL^C_n(P)), \mathbf{\omega}). \tag{${\rm CP}_{\mathbf{\omega}}$}
\end{equation}
\item The OWA Bayes Premium is defined as:
\begin{equation}\label{eq:owabayes}
\BP_{\mathbf{\omega}}(\mathbf{x}) \in  \arg\dmin_{P \in \mathbb{P}}   \OWA( (c_1L^B_1(P), \ldots, c_nL^B_n(P)), \mathbf{\omega}). \tag{${\rm BP}_{\mathbf{\omega}}$}
\end{equation}
where $\mathbf{x}=(x_1, \ldots, x_t)$ is a given sample of claims for the policyholder.
\end{itemize}
\end{defn}

For the sake of simplicity, we use the following notation for the OWA functions involved in the optimization problems of \eqref{eq:owacol} and \eqref{eq:owabayes}:

\begin{align*}
\mathcal{L}^C_{OWA}(P) &= \OWA( (c_1L^C_1(P), \ldots, c_nL^C_n(P)), \mathbf{\omega}) = \dsum_{i=1}^n c_i \omega_{(i)} L^C_i(P),\\
\mathcal{L}^B_{OWA}(P) &= \OWA( (c_1L^B_1(P), \ldots, c_nL^B_n(P)), \mathbf{\omega}) = \dsum_{i=1}^n c_i \omega_{(i)} L^B_i(P),
\end{align*}
for any $P \in \R$.

Note that in the above collective and Bayes OWA premiums the $\omega$--weights are assigned to sorted losses while the $\mathbf{c}$--weights are assigned to experts. Then, each component of $\mathbf{c}$ indicates the information about the specific expert (as the confidence level) and might be fixed either by the decision maker or computed by using any of the techniques for linear opinion pools \cite{bacharach,rufo-perez,stone}. On the other hand, $\omega$ determines the OWA operator used to sort the losses. Observe also that the objective functions of the problems that need to be minimized  are, in general, not convex or differentiable functions, so the standard optimality conditions cannot be applied.

Through this paper, we analyze both collective and Bayes premiums from a similar perspective because both loss functions have a similar structure. Then, unless necessary, for each $i=1, \ldots, n$, we denote by $E_i[X]$ (resp. $E_i[X^2]$) the \textit{prior} or \textit{posterior} expectation (resp. second order moment) of $X$, which are part of the coefficients of the corresponding loss functions. Note that in the collective case, $E_i[X^p]= E_{\pi_i}[E_\theta[X^p]]$, while in the Bayes case, $E_i[X^p] = E_{\pi_i(\theta|x)}[E_\theta[X^p]]$, for $p=1, 2$.

\begin{remark}
\label{rem:1}
The sum case ($\omega_i=1$ for all $i=1, \ldots, n$) is a very special case in this framework (the reader may observe that since all the $\omega$ weights are equal, no matter the sorted sequence of the losses). The overall loss function is
$$
\mathcal{L}^\ell_{OWA}(P) = \dsum_{i=1}^n c_i L_i^\ell (P),
$$
for $\ell \in \{C, B\}$.

This function is convex and  continuously differentiable. Hence, applying optimality conditions, we get that its minimum is attained at:
$$
P^* = \dsum_{i=1}^n c_i E_i[X],
$$
that is, the $\mathbf{c}$--weighted mean of the standard collective or Bayes premiums for each of the experts.
\end{remark}

In order to compute the OWA collective and Bayes premiums, for general $\omega$ weights, we first provide a mathematical programming formulation that allows to solve, using commercial solvers, the desired premiums.
\begin{prop}
\label{prop:1}
The collective and Bayes OWA premiums can be computed by solving the following mathematical programming problem:
\begin{align}
& \min  \; \dsum_{j=1}^n \omega_j z_j\label{eq:OWAP}\tag{${\rm OWAP}$}\\
s.t. \quad &     y_i \geq c_i\,L^\ell_i (P), & \forall i=1, \ldots, n,\label{ctr:1}\\
& y_i \leq z_j+ M (1-w_{ij}), & \forall i,j =1, \ldots, n,\label{ctr:2}\\
& z_j \geq z_{j+1}, & \forall j=1, \ldots, n-1,\label{ctr:3}\\
& \dsum_{i=1}^n w_{ij} = 1, & \forall j=1, \ldots, n,\label{ctr:3b}\\
& y_i, z_j, P \geq 0, &\forall i, j=1, \ldots, n,\label{ctr:4}\\
& w_{ij} \in \{0,1\}, & \forall i, j =1, \ldots, n.\label{ctr:5}
\end{align}
with $\ell=C$ for the collective OWA premium and $\ell=B$ for the Bayes OWA premium and where $M>0$ is a large enough constant.
\end{prop}

\begin{proof}
Let $P$ be a premium. Denote by $y_i$ the marginal contribution of the $i$th--expert to the overall (aggregated loss), i.e., $y_i = c_i\,L^\ell_i (P)$ (constraint \eqref{ctr:1}). Let $z_j$ denote the $y$--value which is sorted in $j$th position when sorted in non increasing order and $w_{ij}$ the binary variable that takes value $1$ if $y_i=z_j$, and $0$ otherwise. Constraints \eqref{ctr:2}--\eqref{ctr:3b} (and the minimization criterion) assure the correct definition of these variables.
\end{proof}
Observe that constraints \eqref{ctr:1} are second order cone constraints, while \eqref{ctr:2} and \eqref{ctr:3} are linear constraints. Because some of the variables are required to be binary, the problem become a mixed integer second order cone optimization problem, that can be solved by many of the available optimization commercial solvers (as Gurobi\cite{gurobi}, IBM ILOG CPLEX\cite{cplex}, etc).

In some cases, the above formulation may be improved avoiding the binary variables, as in the case when $\omega_1\ge \ldots \ge \omega_n \ge 0$, which is refereed in the literature as the ``convex case'' (see \cite{nickel-puerto}).

\begin{prop}
\label{prop:2}
If $\omega_1 \geq \cdots \geq \omega_n \geq 0$, then, \eqref{eq:OWAP} is equivalent to the following continuous mathematical programming problem:

\begin{align}
\min \dsum_{j=1}^n v_j+ \dsum_{i=1}^n w_i \nonumber\\
\mbox{ s.t. } & y_j \geq c_j\,L^\ell_j (P), & \forall j=1, \ldots, n,\nonumber\\
& v_j + w_i \geq \omega_i y_j,& \forall i, j =1, \ldots, n,\nonumber\\
& y_i, z_i \geq 0, &forall i=1, \ldots, n\nonumber\\
& v_i, w_i \in \R, &\forall i=1, \ldots, n,\nonumber\\
& P \in  R\nonumber.
\end{align}
for $\ell \in \{C,B\}$.
\end{prop}

\begin{proof}
The proof follows by the representation of the ordering between the residuals by permutation variables, which for $\lambda_1 \geq \cdots \geq \lambda_n\geq 0$, allows to write the objective function in \eqref{eq:OWAP} as an assignment problem which is totally unimodular, so it can be equivalently rewritten using its dual problem. The interested reader is refereed to \cite{BPE14} for further details on this transformation.
\end{proof}

In the case when the $\mathbf{\omega}$--weights are nonnegative, the OWA (collective and Bayes) premiums can be, constructively, computed using a finite search strategy. 

Let us now denote by $\mathcal{A}_0 = \{0=p_0 \leq p_1 < p_2 <\cdots < p_k < p_{k+1} = p_k + 1\}$ the solutions of the equations $c_iL^\ell_i (P)=c_jL^\ell_j (P)$ for all $i<j, \; i,j=1,\ldots,n$. Observe that $\mathcal{A}_0$ is the set of points where the change of ordering in the losses is possible. Let us denote by $I_j = (p_{j-1}, p_{j})$ the open interval between two consecutive points in $\mathcal{A}_0$.

Let us denote by $\mathcal{A}_\omega$ the set of critical points of $L^\ell_{OWA}$ inside all the open intervals, i.e.,
$$
\mathcal{A}_\omega = \bigcup_{i=1}^n \{P^* \in \arg\min L^\ell_{OWA}(P): P \in I_j\}
$$
and $\mathcal{A} = \mathcal{A}_0 \cup \mathcal{A}_\mathbf{\omega}$.

With the above definitions, Algorithm \ref{alg:1} shows a procedure to compute the premiums.

\begin{algorithm}[h]
\SetKwInOut{Output}{Output}
\KwData{Losses: $L_1^\ell, \ldots, L_n^\ell$; $\mathbf{\omega}$; $c$; $\mathcal{A}=\emptyset$}

{\it $\star$ Step 1.}\\

\For{$i, j \in \{1, \ldots, n\}$ $i<j$}{Compute the breaking points of each pair of expert's losses by solving $c_i\left(P^2 -2 E_i[X] P + E_i[X^2]\right) = c_j\left(P^2 -2 E_j[X] P + E_j[X^2]\right)$: $\mathcal{A}_0$.}

Sort $\mathcal{A}_0 = \{p_0=0 \leq p_1 < p_2 < \cdots < p_k < p_{k} + 1\}$.

{\it $\star$ Step 2.}\\

\For{$j \in \{1, \ldots, k\}$}{
Compute $x_j=\dfrac{p_{j-1}+p_{j}}{2}$.

Sort $L^\ell_i(x_j)$ in non increasing order through a permutation of the indices $\sigma$.

Set $\mathcal{L}^\ell_j (P) = \dsum_{i=1}^n \omega_i L^\ell_{\sigma(i)} (P)$.

$\mathcal{A}_\mathbf{\omega} = \mathcal{A}_\mathbf{\omega} \cup \{\arg\min \mathcal{L}^\ell_j(P)\}$.

}

{\it $\star$ Step 3.}\\

\Output{$\mathcal{L}^* = \min\{\mathcal{L}^\ell_1(P_1), \ldots, \mathcal{L}^\ell_k(P_k)\} = \mathcal{L}^\ell_{j^*}(P_{j^*})$.\\
$P^* = P_{j^*}$.
}
\caption{Procedure to compute OWA premiums for $\omega_1, \ldots, \omega_n\geq 0$.\label{alg:1}}
\end{algorithm}

\begin{theo}\label{theo:0}
If $\omega_1, \ldots, \omega_n \geq 0$, Algorithm \ref{alg:1}  computes the optimal OWA premiums in $O(n^2)$.
\end{theo}
\begin{proof}
The details of the proof are given in Appendix \ref{appendix}.
\end{proof}

The following example illustrates the usage of the proposed algorithm.

\begin{ex}
\label{ex:1}
Let us assume that three experts (with the same confidence level, i.e., $c=(\frac{1}{3},\frac{1}{3},\frac{1}{3})$) have manifested their opinion over a portfolio of policyholders. With such an information, we get that the collective loss functions for each of them are:
\begin{align*}
L_1^C(P) = P^2 - 2P+2,\\
L_2^C(P) = P^2 - 4P+6,\\
L_3^C(P) = P^2 - 6P+12.
\end{align*}

\begin{figure}[h]
\centering
\begin{tikzpicture}[scale=0.8]
  \begin{axis}[
xmax=4,xmin=0,
ymin=0,ymax=15,
xlabel=$P$,ylabel=$\mathcal{L}$,
xtick={0,...,4},
]
   \addplot[dotted, domain=0:4]{x^2 - 2*x + 2};
   \addplot[dashed, domain=0:4]{x^2 - 4*x + 6};
   \addplot[ domain=0:4] {x^2 - 6*x + 12};
   \end{axis}
\end{tikzpicture}
\caption{Loss functions in Example \ref{ex:1} (${L}_1^C$: dotted, ${L}_2^C$: dashed, ${L}_3^C$: black).\label{fig:1}}
\end{figure}
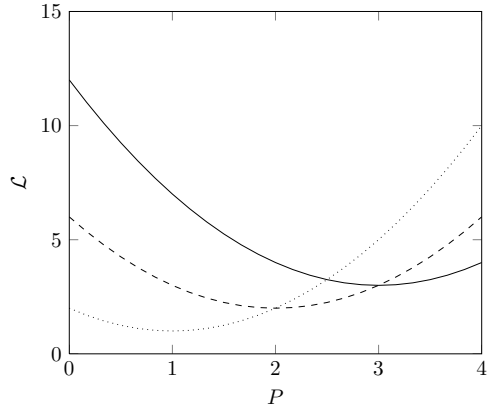

 The three functions are drawn in Figure \ref{fig:1}. Following the procedure described in Algorithm \ref{alg:1}, first we get that $\mathcal{A}_0= \{0, 2, 2.5, 3, 4\}$ (the intersection of each pair of the above three curves and the extreme points $0$ and $4$).

The second part of the algorithms needs to fix the $\omega$--weights. In Table \ref{t:1} we show the results obtained for different choices of $\omega$. In particular, we show the results obtained for the classical (weighted) SUM, MAXimum and MINimum criterions, the Anti K--center (aKC) criterion that minimizes the sum of the two smallest expected losses (obviating as outlier the largest one), and two cases of the Hurwicz criterion which minimizes the weighted sum of the maximum and the minimum losses (a weighted combination of the optimistic and the pessimistic principle). For each of the choices, the procedure consists of finding the minimum premium between each consecutive pair of elements in $\mathcal{A}_0$. For instance, for the MAX case ($\omega=(1,0,0)$), in the first interval $I_1=(0,2)$, the loss functions verify $L_3^C(P) \geq L_2^C(P) \geq L_1^C(P)$ (being the permutation of the indices $\sigma=[3,2,1]$). Hence:

\begin{align*}
\mathcal{L}_{OWA}(P) &= \omega_1 \times c_3 L_3^C(P) + \omega_2 \times c_2 L_2^C(P) + \omega_3 \times c_1 L_1^C(P)\\
&=   1 \times \dfrac{1}{3} \;L_3^C(P) + 0 \times \dfrac{1}{3}\; L_2^C(P) + 0 \times \dfrac{1}{3} L_1^C(P)\\
&= \dfrac{1}{3} P^2 - 2 P + 4,
\end{align*}
whose minimum over $[0,2]$ is $P^*=2$ (observe from Figure \ref{fig:1} that the quadratic function is monotone in the interval so reaching  its minimum in one of the extremes of the interval). The loss associated to $P^*$ is $\mathcal{L}_{OWA}(P^*)=1.3$. The rest of the analyses for the intervals $(2,2.5)$, $(2.5,3)$ and $(3,4)$, are similar, reaching the mininum premiums: $2.5$, $2.5$ and $3$, with losses: $1.1$, $1.1$ and $1.7$, respectively. Hence, the minimum loss among the obtained values is $1.1$ associated with the premium $P^*=2.5$, so the OWA premium for the MAX case is $2.5$.

Table \ref{t:1} shows the OWA-weights choice (first two columns), the interval $I_j=(p_j, p_{j+1})$ where the order of the functions changes, the permutation $\sigma$ that sorts in non increasing order the losses at $I_j$, the aggregated expression of the loss function at each of the intervals $I_j$ ($\mathcal{L}_{OWA}(P)$), the best premium at each of the intervals ($P^*$) and the approximate minimal value of the loss function at that premium ($\mathcal{L}_{OWA}(P^*)$). Numbers marked in bold face are the minimum values of the loss function, which allow us to select, accordingly, the premium. Observe that from Remark \ref{rem:1}, the algorithm is not needed to be applied to the SUM case, but it is shown for illustrative purposes. Figure \ref{fig:2} shows the OWA loss functions.
\end{ex}

\begin{table}[!]
\renewcommand{\arraystretch}{1.2}
\centering
\begin{tabular}{|c|c|c|c|c|c|c|}\hline
\multicolumn{2}{|c|}{$\mathbf{\omega}$} & $I_j=(p_j, p_{j+1})$ & $\sigma$ & $\mathcal{L}_{OWA}(P)$ & $P^*$ & $\mathcal{L}_{OWA}(P^*)$\\
\hline \parbox[t]{2mm}{\multirow{4}{*}{\rotatebox[origin=c]{90}{\textrm{\scriptsize SUM}}}} & \multirow{4}{*}{$(1, 1, 1)$}   & $(0, 2)$ &  $[3, 2, 1]$ &  $P^2-4P+\frac{20}{3}$ &$2$ & $\mathbf{2.7}$\\
& & $(2, 2.5)$ &  $[3, 1, 2]$ &  $P^2-4P+\frac{20}{3}$ &$2$ & $\mathbf{2.7}$\\
& & $(2.5, 3)$ &  $[1, 3, 2]$ &  $P^2-4P+\frac{20}{3}$ &$2.5$ & $2.9$\\
& & $(3, 4)$ &  $[1, 2, 3]$ &  $P^2-4P+\frac{20}{3}$ &$3.$ & $3.7$\\
\hline \parbox[t]{2mm}{\multirow{4}{*}{\rotatebox[origin=c]{90}{\textrm{\scriptsize MAX}}}} &\multirow{4}{*}{$(1, 0, 0)$}   & $(0, 2)$ &  $[3, 2, 1]$ &  $\frac{1}{3} P^2-2P+4$ &$2$ & $1.3$\\
& & $(2, 2.5)$ &  $[3, 1, 2]$ &  $\frac{1}{3} P^2-2P+4$ &$2.5$ & $\mathbf{1.1}$\\
& & $(2.5, 3)$ &  $[1, 3, 2]$ &  $\frac{1}{3} P^2-\frac{2}{3}P+\frac{2}{3}$ &$2.5$ & $\mathbf{1.1}$\\
& & $(3, 4)$ &  $[1, 2, 3]$ &  $\frac{1}{3} P^2-\frac{2}{3}P+\frac{2}{3}$ &$3.$ & $1.7$\\
\hline \parbox[t]{2mm}{\multirow{4}{*}{\rotatebox[origin=c]{90}{\textrm{\scriptsize MIN}}}} &\multirow{4}{*}{$(0, 0, 1)$}   & $(0, 2)$ &  $[3, 2, 1]$ &  $\frac{1}{3} P^2-\frac{2}{3}P+\frac{2}{3}$ &$1.$ & $\mathbf{0.33}$\\
& & $(2, 2.5)$ &  $[3, 1, 2]$ &  $\frac{1}{3} P^2-\frac{4}{3}P+2$ &$2$ & $0.67$\\
& & $(2.5, 3)$ &  $[1, 3, 2]$ &  $\frac{1}{3} P^2-\frac{4}{3}P+2$ &$2.5$ & $0.75$\\
& & $(3, 4)$ &  $[1, 2, 3]$ &  $\frac{1}{3} P^2-2P+4$ &$3.$ & $1.$\\
\hline \parbox[t]{2mm}{\multirow{4}{*}{\rotatebox[origin=c]{90}{\textrm{\scriptsize akC}}}} &\multirow{4}{*}{$(0, 1, 1)$}   & $(0, 2)$ &  $[3, 2, 1]$ &  $\frac{2}{3}P^2-2P+\frac{8}{3}$ &$1.5$ & $\mathbf{1.2}$\\
& & $(2, 2.5)$ &  $[3, 1, 2]$ &  $\frac{2}{3}P^2-2P+\frac{8}{3}$ &$2$ & $1.3$\\
& & $(2.5, 3)$ &  $[1, 3, 2]$ &  $\frac{2}{3}P^2-\frac{10}{3}P+6$ &$2.5$ & $1.8$\\
& & $(3, 4)$ &  $[1, 2, 3]$ &  $\frac{2}{3}P^2-\frac{10}{3}P+6$ &$3.$ & $2$\\
\hline \parbox[t]{2mm}{\multirow{4}{*}{\rotatebox[origin=c]{90}{\textrm{\scriptsize HURWICZ$_{0.5}$}}}} &\multirow{4}{*}{$(0.5, 0, 0.5)$}   & $(0, 2)$ &  $[3, 2, 1]$ &  $\frac{1}{3} P^2-\frac{4}{3}P+\frac{7}{3}$ &$2$ & $1.$\\
& & $(2, 2.5)$ &  $[3, 1, 2]$ &  $\frac{1}{3} P^2-\frac{5}{3}P+3$ &$2.5$ & $\mathbf{0.92}$\\
& & $(2.5, 3)$ &  $[1, 3, 2]$ &  $\frac{1}{3} P^2-P+\frac{4}{3}$ &$2.5$ & $\mathbf{0.92}$\\
& & $(3, 4)$ &  $[1, 2, 3]$ &  $\frac{1}{3} P^2-\frac{4}{3}P+\frac{7}{3}$ &$3.$ & $1.3$\\
\hline \parbox[t]{2mm}{\multirow{4}{*}{\rotatebox[origin=c]{90}{\textrm{\scriptsize HURWICZ$_{0.7}$}}}} &\multirow{4}{*}{$
(0.3, 0, 0.7)$}   & $(0, 2)$ &  $[3, 2, 1]$ &  $\frac{1}{3}P^2-1.06P+1.66$ &$1.6$ & $\mathbf{0.81}$\\
& & $(2, 2.5)$ &  $[3, 1, 2]$ &  $\frac{1}{3}P^2-1.53P+2.6$ &$2.3$ & $0.84$\\
& & $(2.5, 3)$ &  $[1, 3, 2]$ &  $\frac{1}{3}P^2-1.13P+1.6$ &$2.5$ & $0.85$\\
& & $(3, 4)$ &  $[1, 2, 3]$ &  $\frac{1}{3}P^2-1.6P+3$ &$3.$ & $1.2$\\
\hline \end{tabular}
\caption{Solutions for different choices of $\omega$--weights in Example \ref{ex:1}.\label{t:1}}
\end{table}

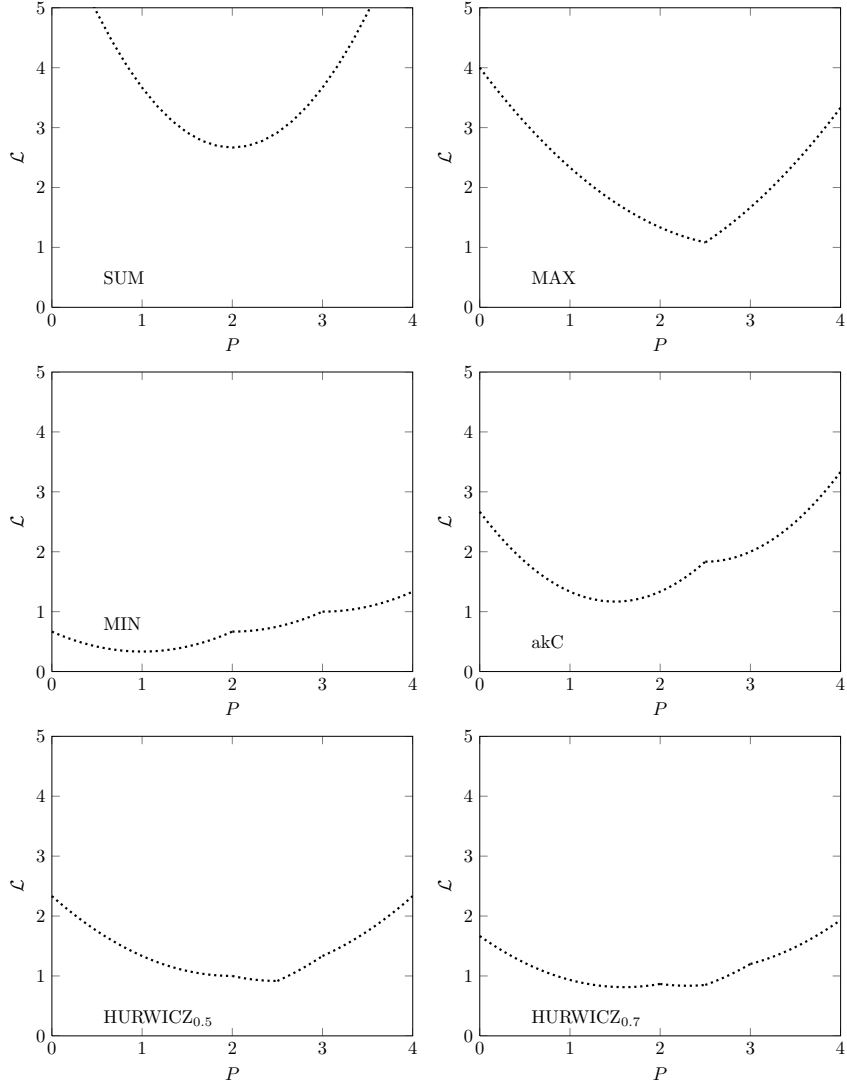
\begin{figure}
\centering
\begin{tikzpicture}[scale=0.7]
  \begin{axis}[
xmax=4,xmin=0,
ymin=0,ymax=5,
xlabel=$P$,ylabel=$\mathcal{L}$,
xtick={0,...,10},
]
   \addplot[ dotted, very thick, domain=0:4] {x^2 - 4*x + (20/3)};
\node[right] at (axis cs:0.5,0.5){SUM};
   \end{axis}
\end{tikzpicture}
\begin{tikzpicture}[scale=0.7]
  \begin{axis}[
xmax=4,xmin=0,
ymin=0,ymax=5,
xlabel=$P$,ylabel=$\mathcal{L}$,
xtick={0,...,10},
]
   \addplot[ dotted, very thick, domain=0:2.5] {(1/3)*x^2 - 2*x + 4};
\addplot[ dotted, very thick, domain=2.5:4] {(1/3)*x^2 - (2/3)*x + (2/3)};
\node[right] at (axis cs:0.5,0.5){MAX};
   \end{axis}
\end{tikzpicture}

\begin{tikzpicture}[scale=0.7]
  \begin{axis}[
xmax=4,xmin=0,
ymin=0,ymax=5,
xlabel=$P$,ylabel=$\mathcal{L}$,
xtick={0,...,10},
]
   \addplot[ dotted, very thick, domain=0:2] {(1/3)*x^2 - (2/3)*x + (2/3)};
\addplot[ dotted, very thick, domain=2:3] {(1/3)*x^2 - (4/3)*x + 2};
\addplot[ dotted, very thick, domain=3:4] {(1/3)*x^2 - 2*x + 4};
\node[right] at (axis cs:0.5,0.8){MIN};
   \end{axis}
\end{tikzpicture}
\begin{tikzpicture}[scale=0.7]
  \begin{axis}[
xmax=4,xmin=0,
ymin=0,ymax=5,
xlabel=$P$,ylabel=$\mathcal{L}$,
xtick={0,...,10},
]
   \addplot[ dotted, very thick, domain=0:2.5] {(2/3)*x^2 - 2*x + (8/3)};
\addplot[ dotted, very thick, domain=2.5:4] {(2/3)*x^2 - (10/3)*x + 6};
\node[right] at (axis cs:0.5,0.5){akC};
   \end{axis}
\end{tikzpicture}

\begin{tikzpicture}[scale=0.7]
  \begin{axis}[
xmax=4,xmin=0,
ymin=0,ymax=5,
xlabel=$P$,ylabel=$\mathcal{L}$,
xtick={0,...,10},
]
   \addplot[ dotted, very thick, domain=0:2] {(1/3)*x^2 - (4/3)*x + (7/3)};
   \addplot[ dotted, very thick, domain=2:2.5] {(1/3)*x^2 - (5/3)*x + 3};
   \addplot[ dotted, very thick, domain=2.5:3] {(1/3)*x^2 - x + (4/3)};
   \addplot[ dotted, very thick, domain=3:4] {(1/3)*x^2 - (4/3)*x + (7/3)};
\node[right] at (axis cs:0.5,0.3){HURWICZ$_{0.5}$};
   \end{axis}
\end{tikzpicture}
\begin{tikzpicture}[scale=0.7]
  \begin{axis}[
xmax=4,xmin=0,
ymin=0,ymax=5,
xlabel=$P$,ylabel=$\mathcal{L}$,
xtick={0,...,10},
]
   \addplot[ dotted, very thick, domain=0:2] {(1/3)*x^2 - 1.06666*x + 1.6666};
   \addplot[ dotted, very thick, domain=2:2.5] {(1/3)*x^2 - 1.53333*x + 2.6};
   \addplot[ dotted, very thick, domain=2.5:3] {(1/3)*x^2 - 1.13333*x + 1.6};
   \addplot[ dotted, very thick, domain=3:4] {(1/3)*x^2 - 1.6*x + 3};
\node[right] at (axis cs:0.5,0.3){HURWICZ$_{0.7}$};
   \end{axis}
\end{tikzpicture}
\caption{OWA loss functions in Example \ref{ex:1}.\label{fig:2}}
\end{figure}

H\"urlimann \cite{hurlimann94} proposed the following desirable properties which should be satisfied by a premium calculation principle:
\begin{itemize}
\item[\bf(P1)] Risk loading: $\mathcal{H}(X) \geq E[X]$.
\item[\bf(P2)] Maximal loss: $\mathcal{H}(X) \leq  {\it ess} \;\;{\it sup}\; X$ .
\item[\bf(P3)] Linear invariance: $\mathcal{H}(aX+b) = a\mathcal{H}(X)+b$, $a, b \geq 0$.
\item[\bf(P4)] Constant risk: $\mathcal{H}(C) = C, \forall C \ge 0$.
\item[\bf(P5)] Additivity: $\mathcal{H}(X+Y)= \mathcal{H}(X)+\mathcal{H}(Y)$,
\end{itemize}
\noindent where $\mathcal{H}$ is a premium principle, $X$ and $Y$ are risks and  {\it ess} {\it sup}\;; $X$  is the essential supremum of $X$. Note that the notation $E[X]$ in the first property (Risk loading) should be adapted adequately to the corresponding premium. For instance, in the collective  (resp. Bayes) premiums, $E[X]$ denotes $E_\pi[E_\theta[X]]$ (resp. $E_{\pi(\theta|x)}[E_\theta[X]]$).

\begin{prop}
The OWA premiums satisfy properties {\bf(P1)}, {\bf(P2)} , {\bf(P3)} and {\bf(P4)}  above.
\end{prop}

\begin{proof}
We prove the result for the collective risk OWA premiums. The Bayes case follows in a similar way:
\begin{itemize}
\item[\bf(P1)] Note that $\dmin_P  \dsum_{i=1}^n c_{(i)} \omega_i L^C_{(i)}(P) \geq \dsum_{i=1}^n c_{(i)} \omega_i \dmin_{P_i} L^C_{(i)}(P_i)$. Hence,  the optimal OWA premium must verify  $P^* \geq E[X]$.
\item[\bf(P2)] Observe that  $\dsum_{i=1}^n c_{(i)} \omega_i L^C_{(i)}(P) \leq{\it ess} {\it sup}\;; $X$ X \dsum_{i=1}^n c_{(i)} \omega_i$, for any premium $P\geq 0$. Thus, $P^* \leq$  {\it ess} {\it sup}\; $X$.
\item[\bf(P3)] Note that if $p_{ij}$ is the intersection point between the curves $\mathcal{L}^C_i$ and $\mathcal{L}_j^C$ that conform the OWA function to be minimized, then, linear transformations of $X$ in the form $aX+b$ produces intersection points in the form $ap_{ij}+b$. Furthermore, once the breakpoints are known, between those breakpoints one must minimize functions in the form $\dsum_{i=1}^n \omega_{\sigma(i)} c_i (P^2 - 2P E_{\pi_i}[E_\theta[X]] + E_{\pi_i}[E_\theta[X]^2])$, for some adequate permutation of the indices. Hence, the possible critical points are in the form:
$$
P^* = \dfrac{\dsum_{i=1}^n \omega_{\sigma(i)} c_i E_{\pi_i}[E_\theta[X]] }{\dsum_{i=1}^n \omega_{\sigma(i)} c_i}.
$$
Performing the same operations for the forms $aX+b$, we get critical points $P^{'} = aP^* + b$. Since for $a, b\geq 0$ the order does not change, we get that $\mathcal{H}(aX+b)= a\mathcal{H}(X)+b$.
\item[\bf(P4)]  In this case, since $X$ is constant $X=C$, $\dsum_{i=1}^n c_{(i)} \omega_i L^C_{(i)}(P) = \dsum_{i=1}^n c_{(i)} \omega_i (P-C)^2 =  (P-C)^2 \dsum_{i=1}^n c_{(i)} \omega_i $, whose unique minimum is $P^*=C$.
\end{itemize}
\end{proof}

In some of the standard premium computation principles, some of the desirable properties are not verified. For instance, the variance principle does not verify {\bf(P2)}, and exponential, Esscher and variance principles do not verify {\bf(P3)}. In the following example we illustrate, that in our case, the additivity property {\bf(P5)}  is not satisfied, in general, by the OWA premiums.

\begin{ex}
\label{counterexample}
Let us consider two independent risks $X$ and $Y$ distributed as Poisson distributions $\Poi(\lambda)$ and $\Poi(\mu)$, respectively. Let also assume that two experts (with the same confidence level) have manifested their opinion about the risks by giving these Gamma \textit{prior} distributions:
\begin{align*}
\lambda \sim \Ga (2, 10), \Ga(2, 20),\\
\mu \sim \Ga(3, 10), \Ga(7, 20).
\end{align*}
Note that, since the \textit{prior} distribution becomes Gamma in this case, the collective loss functions for each of the experts and risks are:
\begin{align*}
X: & L^C_1 = P^2-0.40P+0.24,\\
 & L^C_2 = P^2-0.20P+0.11,\\
 Y: & L^C_1 = P^2-0.60P+0.39,\\
 & L^C_2 = P^2-0.70P+0.4725.
 \end{align*}

 Hence, the OWA collective premiums under the MAX criterion ($\mathbf{\omega}=(1,0))$ are reached at $P_X^* = 0.2$ and $P_Y^* = 0.35$. On the other hand, by additivity of the Poisson and Gamma distributions (when the second parameters of the Gamma distributions coincide), we get that the collective risk functions for the risk $X+Y$ are, for each of the experts:
 \begin{align*}
X+Y: & L^C_1 =P^2-P+0.75,\\
 & L^C_2 = P^2-0.90P+0.6525,
 \end{align*}
 and the OWA risk premium for the same weight is obtained at $P_{X+Y}^* = 0.5 \neq P_X^* + P_Y^* = 0.55$.

 In Figure \ref{fig:counterexample} we draw the loss functions for each of the risks $X$, $Y$ and $X+Y$ for the maximum criterion, where the reader can check the obtained results. 
  \end{ex}

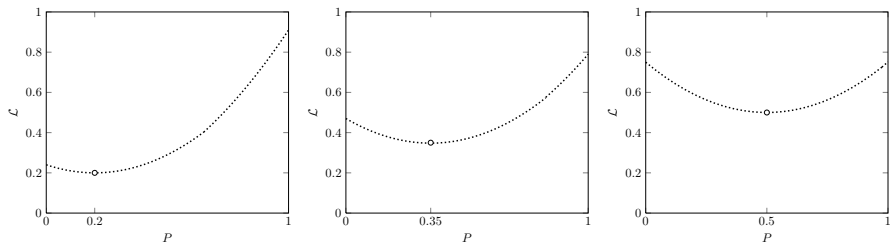
\begin{figure}[!]
\centering
\begin{tikzpicture}[scale=0.47]
  \begin{axis}[
xmax=1,xmin=0,
ymin=0,ymax=1,
xlabel=$P$,ylabel=$\mathcal{L}$,
xtick={0,0.2,1},
]
   \addplot[dotted, very thick, domain=0:0.65]{x^2-0.40*x+0.24};
   \addplot[dotted, very thick, domain=0.65:1]{x^2-0.20*x+0.11};

   \addplot[color=black, solid, mark=*, mark options={fill=white}]  coordinates {
   (0.2,0.2)
   };
   \node at (axis cs:4,11){X};
   \end{axis}
\end{tikzpicture}
\begin{tikzpicture}[scale=0.47]
  \begin{axis}[
xmax=1,xmin=0,
ymin=0,ymax=1,
xlabel=$P$,ylabel=$\mathcal{L}$,
xtick={0,0.35,1},
]
   \addplot[dotted, very thick, domain=0:0.825]{x^2-0.70*x+.47};
   \addplot[dotted, very thick, domain=0.825:1]{x^2-0.60*x+.39};


   \addplot[color=black, solid, mark=*, mark options={fill=white}]  coordinates {
   (0.35, 0.35)
   };
   \node at (axis cs:4,11){Y};
   \end{axis}
\end{tikzpicture}
\begin{tikzpicture}[scale=0.47]
  \begin{axis}[
xmax=1,xmin=0,
ymin=0,ymax=1,
xlabel=$P$,ylabel=$\mathcal{L}$,
xtick={0,0.5,1},
]

   \addplot[dotted, very thick, domain=0:0.975]{x^2-x+0.75};
   \addplot[dotted, very thick, domain=0.975:1]{x^2-0.9000*x+0.6525};

   \addplot[color=black, solid, mark=*, mark options={fill=white}]  coordinates {
   (0.5, 0.5)
   };
   \node at (axis cs:4,11){$X+Y$};
   \end{axis}
\end{tikzpicture}
\caption{Loss functions in Example \ref{counterexample} ($X$, $Y$ and $X+Y$ from the left to the right).\label{fig:counterexample}}
\end{figure}

Observe that, in general, non additivity of the computation principle is reasonable from our construction of the premiums. An expert may express pessimistic opinion over a risk while the same expert may manifest optimism for another risk. Hence, there is no reason to think that if the addition of the two risk is performed, the best premium is obtained under the same sorting for both risks.

\section{Application to real data set}
\label{sec:3}

In this section we report the results of applying our proposed methodology to a real data set in order to both show the behavior of the new premiums and also to illustrate the differences between the results obtained with our approach and those obtained with the standard approach.

We consider six different choices for the $\mathbf{\omega}$--weights (see Table \ref{weights}) and the following combinations of risk-prior distributions: Poisson-Gamma and Geometric-Beta (see Proposition \ref{prop:poigamma}). We compute the OWA Bonus--Malus premiums for the time horizon $t=0, \ldots, 4$ with number of claims ranging in $k=0, \ldots, 4$. The parameters of the experts' prior distributions¡, computed by the maximum likelihood method, are based on the data in \cite{morillo03} (Portfolio 1) and \cite{boucher06} (Portfolio 2), the two insurance portfolios shown in Table \ref{table:portfolios}. Those parameters are estimated as follows: experts $\#1$ and $\#2$ base their knowledge in Portfolios 1 and 2, respectively; expert $\#3$ is more conservative, being pessimistic about the driving of the policyholders; and expert $\#4$ is slightly optimistic and less conservative.  The \textit{prior} parameters for each of the experts and for each of the Gamma and Beta distributions are shown in Table \ref{table:parameters}.

\begin{table}[h]
\centering
\small
\begin{subtable}{.4\textwidth}
\begin{tabular}{|c|c|}\hline
 & $\mathbf{\omega}$\\\hline
SUM & $(1,1,1,1)$\\
MAX & $(1,0,0,0)$\\
MIN & $(0,0,0,1)$\\
akC & $(0,0,1,1)$\\
HURWICZ$_{0.5}$ & $(0.5,0,0,0.5)$\\
HURWICZ$_{0.7}$ & $(0.3,0,0,0.7)$\\\hline
\end{tabular}
\caption{OWA weights.\label{weights}}
\end{subtable}%
\begin{subtable}{0.4\textwidth}
\begin{tabular}{|c|c|c|}\hline
$k$ & Portfolio 1 & Portfolio 2\\\hline
0 & 122618 & 371481 \\
1 & 21686 & 26784 \\
2 & 4014 & 2118 \\
3 & 832 & 174 \\
4 & 224 & 18 \\
5 & 68 & 2 \\
6 & 17 & 2 \\
7 & 7 & 0 \\
$\ge$ 8 & 7 & 0 \\
\hline
\end{tabular}
\caption{Insurance portfolios.\label{table:portfolios}}
\end{subtable}
\caption{Input data for the experiments.\label{t:2}}
\end{table}


\begin{table}[h]
\centering
\begin{tabular}{|c|c|c|c|c||c|c|c|c|}\cline{2-9}
\multicolumn{1}{c|}{} & \multicolumn{4}{c}{Gamma$(\alpha_i, \beta_i)$} & \multicolumn{4}{|c|}{Beta$(\alpha_i, \beta_i)$}\\\hline
Parameters& \#1 & \#2 &  \#3 &  \#4 &  \#1 &  \#2 &  \#3 &  \#4\\\hline
$\alpha_i$ & 0.77 & 0.68 & 2.1 & 0.4 & 30.59 & 66.83 & 321.5 & 2.1 \\
$\beta_i$ & 3.40 & 9.85 & 15  & 3.1 & 6.66 & 4.56 & 9.3  & 3.2 \\\hline
\end{tabular}
\caption{Estimated experts' prior distribution parameters.\label{table:parameters}}
\end{table}

In Tables \ref{lemaire:PG} and \ref{PG} we show the Bonus--Malus premiums for the Poisson--Gamma combination. In Table \ref{lemaire:PG} we report the results by applying the classical Lemaire's approach for each of the four experts. The reader may  observe that expert $\#3$ charges higher premiums to the ``good'' drivers (those with less number of claims) and lower premiums to the ``bad'' ones (those with higher number of claims) than other experts.  This expert is, as expected, very conservative (recall that most of the policyholders are good drivers). On the other hand, expert $\#4$ loads lowest premiums to good drivers and highest premiums to the clients with high number of claims. In Table \ref{PG} we show the OWA Bonus--Malus premiums. In this case, the highest ``discounts'' to the good drivers are obtained with the MAX criterion (those clients not reporting any claim in its first contract year are loaded $22.70\%$ less than their initial premium). On the other hand, the MIN criterion overloads bad drivers (those clients with $4$ claims in the $4$th period are charged $363.85\%$ over their initial premium). The criterion for which the good drivers are more penalized is the akC criterion (the deduction from the first free--claim period is $7.23\%$). Observe that our approach smooths the pessimistic and optimistic information from the experts $\#3$ and $\#4$.

\begin{table}[!]
\centering
\begin{tabular}{|c|c|ccccc|}
\hline\parbox[t]{2mm}{\multirow{6}{*}{\rotatebox[origin=c]{90}{\textrm{\scriptsize EXPERT \#1}}}}& \backslashbox{$t$}{$k$} &0 	 & 	1 	 & 	2 	 & 	3 	 & 	4\\ \cline{2-7}
& 	 0 	 & 	 100 & 	 & 	 & 	 & 	\\
& 	 1 	 & 	 77.3001 	 &  	178.1350 	 &  	278.9700 	 &  	379.8049 	 &  	480.6399 	  \\
& 	 2 	 & 	 62.9993 	 &  	145.1794 	 &  	227.3595 	 &  	309.5396 	 &  	391.7198 	  \\
& 	 3 	 & 	 53.1638 	 &  	122.5139 	 &  	191.8640 	 &  	261.2141 	 &  	330.5642 	  \\
& 	 4 	 & 	 45.9846 	 &  	105.9698 	 &  	165.9550 	 &  	225.9401 	 &  	285.9253 	  \\
\hline\parbox[t]{2mm}{\multirow{6}{*}{\rotatebox[origin=c]{90}{\textrm{\scriptsize EXPERT \#2}}}}& \backslashbox{$t$}{$k$} &0 	 & 	1 	 & 	2 	 & 	3 	 & 	4\\ \cline{2-7}
& 	 0 	 & 	 100 & 	 & 	 & 	 & 	\\
& 	 1 	 & 	 90.7898 	 &  	223.8571 	 &  	356.9244 	 &  	489.9917 	 &  	623.0590 	  \\
& 	 2 	 & 	 83.1331 	 &  	204.9783 	 &  	326.8234 	 &  	448.6686 	 &  	570.5137 	  \\
& 	 3 	 & 	 76.6674 	 &  	189.0360 	 &  	301.4046 	 &  	413.7732 	 &  	526.1418 	  \\
& 	 4 	 & 	 71.1349 	 &  	175.3946 	 &  	279.6544 	 &  	383.9142 	 &  	488.1740 	  \\
\hline\parbox[t]{2mm}{\multirow{6}{*}{\rotatebox[origin=c]{90}{\textrm{\scriptsize EXPERT \#3}}}}& \backslashbox{$t$}{$k$} &0 	 & 	1 	 & 	2 	 & 	3 	 & 	4\\ \cline{2-7}
& 	 0 	 & 	 100 & 	 & 	 & 	 & 	\\
& 	 1 	 & 	 93.7500 	 &  	138.3929 	 &  	183.0357 	 &  	227.6786 	 &  	272.3214 	  \\
& 	 2 	 & 	 88.2353 	 &  	130.2521 	 &  	172.2689 	 &  	214.2857 	 &  	256.3025 	  \\
& 	 3 	 & 	 83.3333 	 &  	123.0159 	 &  	162.6984 	 &  	202.3810 	 &  	242.0635 	  \\
& 	 4 	 & 	 78.9474 	 &  	116.5414 	 &  	154.1353 	 &  	191.7293 	 &  	229.3233 	  \\
\hline\parbox[t]{2mm}{\multirow{6}{*}{\rotatebox[origin=c]{90}{\textrm{\scriptsize EXPERT \#4}}}}& \backslashbox{$t$}{$k$} &0 	 & 	1 	 & 	2 	 & 	3 	 & 	4\\ \cline{2-7}
& 	 0 	 & 	 100 & 	 & 	 & 	 & 	\\
& 	 1 	 & 	 75.6098 	 &  	264.6341 	 &  	453.6585 	 &  	642.6829 	 &  	831.7073 	  \\
& 	 2 	 & 	 60.7843 	 &  	212.7451 	 &  	364.7059 	 &  	516.6667 	 &  	668.6275 	  \\
& 	 3 	 & 	 50.8197 	 &  	177.8689 	 &  	304.9180 	 &  	431.9672 	 &  	559.0164 	  \\
& 	 4 	 & 	 43.6620 	 &  	152.8169 	 &  	261.9718 	 &  	371.1268 	 &  	480.2817 	  \\\hline
\end{tabular}
\caption{Standard single--experts' premiums by using Lemaire's approach (Poisson--Gamma).\label{lemaire:PG}}
\end{table}

\begin{table}[!]
\centering
\begin{tabular}{|c|c|ccccc|}
\hline\parbox[t]{2mm}{\multirow{6}{*}{\rotatebox[origin=c]{90}{\textrm{\scriptsize SUM}}}}& \backslashbox{$t$}{$k$} &0 	 & 	1 	 & 	2 	 & 	3 	 & 	4\\ \cline{2-7}
& 	 0 	 & 	 100 & 	 & 	 & 	 & 	\\
& 	 1 	 & 	 82.6582 	 &  	193.6878 	 &  	304.7174 	 &  	415.7470 	 &  	526.7766 	  \\
& 	 2 	 & 	 71.2369 	 &  	164.2918 	 &  	257.3467 	 &  	350.4016 	 &  	443.4565 	  \\
& 	 3 	 & 	 63.0118 	 &  	143.4898 	 &  	223.9678 	 &  	304.4458 	 &  	384.9238 	  \\
& 	 4 	 & 	 56.7340 	 &  	127.8561 	 &  	198.9782 	 &  	270.1003 	 &  	341.2223 	  \\
\hline\parbox[t]{2mm}{\multirow{6}{*}{\rotatebox[origin=c]{90}{\textrm{\scriptsize MAX}}}}& \backslashbox{$t$}{$k$} &0 	 & 	1 	 & 	2 	 & 	3 	 & 	4\\ \cline{2-7}
& 	 0 	 & 	 100 & 	 & 	 & 	 & 	\\
& 	 1 	 & 	 77.3001 	 &  	178.1350 	 &  	264.2003 	 &  	334.3071 	 &  	399.4045 	  \\
& 	 2 	 & 	 62.9993 	 &  	145.1794 	 &  	222.7198 	 &  	282.8075 	 &  	336.5723 	  \\
& 	 3 	 & 	 53.1638 	 &  	122.5139 	 &  	191.8640 	 &  	245.4111 	 &  	293.9029 	  \\
& 	 4 	 & 	 45.9846 	 &  	105.9698 	 &  	165.9550 	 &  	217.2575 	 &  	262.4762 	  \\
\hline\parbox[t]{2mm}{\multirow{6}{*}{\rotatebox[origin=c]{90}{\textrm{\scriptsize MIN}}}}& \backslashbox{$t$}{$k$} &0 	 & 	1 	 & 	2 	 & 	3 	 & 	4\\ \cline{2-7}
& 	 0 	 & 	 100 & 	 & 	 & 	 & 	\\
& 	 1 	 & 	 90.7898 	 &  	279.9272 	 &  	370.2263 	 &  	460.5253 	 &  	550.8244 	  \\
& 	 2 	 & 	 83.1331 	 &  	263.4609 	 &  	348.4482 	 &  	433.4356 	 &  	518.4230 	  \\
& 	 3 	 & 	 76.6674 	 &  	248.8242 	 &  	329.0900 	 &  	409.3559 	 &  	489.6217 	  \\
& 	 4 	 & 	 71.1349 	 &  	235.7281 	 &  	311.7695 	 &  	387.8108 	 &  	463.8522 	  \\
\hline\parbox[t]{2mm}{\multirow{6}{*}{\rotatebox[origin=c]{90}{\textrm{\scriptsize akC}}}}& \backslashbox{$t$}{$k$} &0 	 & 	1 	 & 	2 	 & 	3 	 & 	4\\ \cline{2-7}
& 	 0 	 & 	 100 & 	 & 	 & 	 & 	\\
& 	 1 	 & 	 92.7707 	 &  	166.6670 	 &  	240.5633 	 &  	314.4597 	 &  	388.3560 	  \\
& 	 2 	 & 	 86.5473 	 &  	154.9738 	 &  	223.4002 	 &  	291.8266 	 &  	360.2530 	  \\
& 	 3 	 & 	 81.1280 	 &  	144.8573 	 &  	208.5866 	 &  	272.3159 	 &  	336.0452 	  \\
& 	 4 	 & 	 76.3628 	 &  	136.0118 	 &  	195.6608 	 &  	255.3099 	 &  	314.9589 	  \\
\hline\parbox[t]{2mm}{\multirow{6}{*}{\rotatebox[origin=c]{90}{\textrm{\scriptsize HURWICZ$_{0.5}$}}}}& \backslashbox{$t$}{$k$} &0 	 & 	1 	 & 	2 	 & 	 3 	 & 	4\\ \cline{2-7}
& 	 0 	 & 	 100 & 	 & 	 & 	 & 	\\
& 	 1 	 & 	 83.6076 	 &  	162.8964 	 &  	242.1853 	 &  	321.4742 	 &  	412.0341 	  \\
& 	 2 	 & 	 72.6757 	 &  	139.4557 	 &  	206.2358 	 &  	273.0159 	 &  	349.6706 	  \\
& 	 3 	 & 	 47.3124 	 &  	122.7064 	 &  	180.6809 	 &  	238.6553 	 &  	303.5528 	  \\
& 	 4 	 & 	 41.8373 	 &  	110.0233 	 &  	161.4229 	 &  	212.8225 	 &  	264.2221 	  \\
\hline\parbox[t]{2mm}{\multirow{6}{*}{\rotatebox[origin=c]{90}{\textrm{\scriptsize HURWICZ$_{0.7}$}}}}& \backslashbox{$t$}{$k$} &0 	 & 	1 	 & 	2 	 & 	 3 	 & 	4\\ \cline{2-7}
& 	 0 	 & 	 100 & 	 & 	 & 	 & 	\\
& 	 1 	 & 	 87.0387 	 &  	154.6070 	 &  	222.1753 	 &  	289.7436 	 &  	376.8512 	  \\
& 	 2 	 & 	 50.0346 	 &  	136.3422 	 &  	194.7450 	 &  	253.1477 	 &  	311.5505 	  \\
& 	 3 	 & 	 44.1294 	 &  	122.8111 	 &  	174.5975 	 &  	226.3839 	 &  	278.1703 	  \\
& 	 4 	 & 	 39.5812 	 &  	112.2283 	 &  	158.9576 	 &  	205.6868 	 &  	252.4160 	  \\\hline
\end{tabular}
\caption{Poisson--Gamma Bonus Malus premiums by using OWA operators.\label{PG}}
\end{table}

In Tables \ref{lemaire:GB} and \ref{GB} we show the results for the Geometric--Beta combination of distributions. In the classical (single--expert) approach (Table \ref{lemaire:GB}), the discounts for good drivers are smaller than those obtained with the Poisson--Gamma distributions, except for expert $\#4$ (optimistic) who charges lower premiums for all the policyholders. Table \ref{GB} shows the results of our approach. Similar to the Poisson--Gamma case, good drivers get better deductions with the MAX criterion. The same occurs with bad drivers obtaining the lowest penalties in the MAX criterium. In this case, the MIN criterion charges lower discounts to good drivers (observe the tiny deduction of 0.31$\%$ for a driver not reporting any claim in its first year). The akC criterion loads the largest premiums to bad drivers (66.46$\%$ for a driver with 4 claims in its $4$th period). In general, premiums under the Geometric--Beta combination are lower, for bad drivers, than those obtained under the Poisson--Gamma combination.

A final observation from our computations is that one may expect that for any of the OWA operators constructed with the experience of a single expert (MIN or MAX), the Bonus--Malus premiums coincide with those computed for any of the single experts using the Lemaire's approach. However, this is not always true. Note that a Bonus--Malus premium is the ratio between the Bayes and the collective premium. Both (absolute) premiums comes from the experience of a single expert. In case both come from the same expert, the Bonus--Malus is the standard Bonus--Malus premium of such an expert. Otherwise, if the collective is constructed from an expert and the Bayes from a different one, the Bonus--Malus premium does not coincide with any of the standard Lemaire's premiums.

Finally, the reader may observe from Tables \ref{PG} and \ref{GB} that the Bonus--Malus premiums give us information about percentage discounts or bonification, but not about the final income by the insurance company. Hence, the final return depends on the initial $(k=0, t=0)$ premium charged to the client. To get a financial equilibrium for all the OWA choices one should compute such an amount that may differ for each of the criteria.

\begin{table}[!]
\centering
\begin{tabular}{|c|c|ccccc|}
\hline\parbox[t]{2mm}{\multirow{6}{*}{\rotatebox[origin=c]{90}{\textrm{\scriptsize EXPERT \#1}}}}& \backslashbox{$t$}{$k$} &0 	 & 	1 	 & 	2 	 & 	3 	 & 	4\\ \cline{2-7}
& 	 0 	 & 	 100 & 	 & 	 & 	 & 	\\
& 	 1 	 & 	 96.7310 	 &  	111.2515 	 &  	125.7719 	 &  	140.2924 	 &  	154.8129 	  \\
& 	 2 	 & 	 93.6689 	 &  	107.7297 	 &  	121.7906 	 &  	135.8514 	 &  	149.9122 	  \\
& 	 3 	 & 	 90.7947 	 &  	104.4241 	 &  	118.0535 	 &  	131.6829 	 &  	145.3123 	  \\
& 	 4 	 & 	 88.0917 	 &  	101.3153 	 &  	114.5390 	 &  	127.7626 	 &  	140.9862 	  \\
\hline\parbox[t]{2mm}{\multirow{6}{*}{\rotatebox[origin=c]{90}{\textrm{\scriptsize EXPERT \#2}}}}& \backslashbox{$t$}{$k$} &0 	 & 	1 	 & 	2 	 & 	3 	 & 	4\\ \cline{2-7}
& 	 0 	 & 	 100 & 	 & 	 & 	 & 	\\
& 	 1 	 & 	 98.5036 	 &  	120.1234 	 &  	141.7432 	 &  	163.3630 	 &  	184.9829 	  \\
& 	 2 	 & 	 97.0513 	 &  	118.3524 	 &  	139.6534 	 &  	160.9545 	 &  	182.2556 	  \\
& 	 3 	 & 	 95.6412 	 &  	116.6328 	 &  	137.6244 	 &  	158.6160 	 &  	179.6075 	  \\
& 	 4 	 & 	 94.2715 	 &  	114.9625 	 &  	135.6534 	 &  	156.3444 	 &  	177.0353 	  \\
\hline\parbox[t]{2mm}{\multirow{6}{*}{\rotatebox[origin=c]{90}{\textrm{\scriptsize EXPERT \#3}}}}& \backslashbox{$t$}{$k$} &0 	 & 	1 	 & 	2 	 & 	3 	 & 	4\\ \cline{2-7}
& 	 0 	 & 	 100 & 	 & 	 & 	 & 	\\
& 	 1 	 & 	 99.6890 	 &  	110.4082 	 &  	121.1274 	 &  	131.8467 	 &  	142.5659 	  \\
& 	 2 	 & 	 99.3798 	 &  	110.0658 	 &  	120.7519 	 &  	131.4379 	 &  	142.1239 	  \\
& 	 3 	 & 	 99.0726 	 &  	109.7256 	 &  	120.3786 	 &  	131.0316 	 &  	141.6845 	  \\
& 	 4 	 & 	 98.7673 	 &  	109.3875 	 &  	120.0076 	 &  	130.6278 	 &  	141.2479 	  \\
\hline\parbox[t]{2mm}{\multirow{6}{*}{\rotatebox[origin=c]{90}{\textrm{\scriptsize EXPERT \#4}}}}& \backslashbox{$t$}{$k$} &0 	 & 	1 	 & 	2 	 & 	3 	 & 	4\\ \cline{2-7}
& 	 0 	 & 	 100& 	 & 	 & 	 & 	\\
& 	 1 	 & 	 52.3810 	 &  	68.7500 	 &  	85.1190 	 &  	101.4881 	 &  	117.8571 	  \\
& 	 2 	 & 	 35.4839 	 &  	46.5726 	 &  	57.6613 	 &  	68.7500 	 &  	79.8387 	  \\
& 	 3 	 & 	 26.8293 	 &  	35.2134 	 &  	43.5976 	 &  	51.9817 	 &  	60.3659 	  \\
& 	 4 	 & 	 21.5686 	 &  	28.3088 	 &  	35.0490 	 &  	41.7892 	 &  	48.5294 	  \\
\hline
\end{tabular}
\caption{Standard single--experts' premiums by using Lemaire's approach (Geometric--Beta).\label{lemaire:GB}}
\end{table}

\begin{table}[!]
\centering
\begin{tabular}{|c|c|ccccc|}
\hline\parbox[t]{2mm}{\multirow{6}{*}{\rotatebox[origin=c]{90}{\textrm{\scriptsize SUM}}}}& \backslashbox{$t$}{$k$} &0 	 & 	1 	 & 	2 	 & 	3 	 & 	4\\ \cline{2-7}
& 	 0 	 & 	 100 & 	 & 	 & 	 & 	\\
& 	 1 	 & 	 56.8821 	 &  	73.1841 	 &  	89.4861 	 &  	105.7881 	 &  	122.0902 	  \\
& 	 2 	 & 	 41.4282 	 &  	52.9390 	 &  	64.4497 	 &  	75.9605 	 &  	87.4713 	  \\
& 	 3 	 & 	 33.4063 	 &  	42.4460 	 &  	51.4858 	 &  	60.5256 	 &  	69.5654 	  \\
& 	 4 	 & 	 28.4515 	 &  	35.9768 	 &  	43.5022 	 &  	51.0275 	 &  	58.5528 	  \\
\hline\parbox[t]{2mm}{\multirow{6}{*}{\rotatebox[origin=c]{90}{\textrm{\scriptsize MAX}}}}& \backslashbox{$t$}{$k$} &0 	 & 	1 	 & 	2 	 & 	3 	 & 	4\\ \cline{2-7}
& 	 0 	 & 	 100 & 	 & 	 & 	 & 	\\
& 	 1 	 & 	 52.3810 	 &  	68.7500 	 &  	85.1190 	 &  	101.4881 	 &  	117.8571 	  \\
& 	 2 	 & 	 35.3471 	 &  	43.5724 	 &  	51.8126 	 &  	60.0603 	 &  	68.3125 	  \\
& 	 3 	 & 	 24.1558 	 &  	29.7362 	 &  	35.3300 	 &  	40.9307 	 &  	46.5353 	  \\
& 	 4 	 & 	 18.4232 	 &  	22.6486 	 &  	26.8868 	 &  	31.1316 	 &  	35.3801 	  \\
\hline\parbox[t]{2mm}{\multirow{6}{*}{\rotatebox[origin=c]{90}{\textrm{\scriptsize MIN}}}}& \backslashbox{$t$}{$k$} &0 	 & 	1 	 & 	2 	 & 	3 	 & 	4\\ \cline{2-7}
& 	 0 	 & 	 100 & 	 & 	 & 	 & 	\\
& 	 1 	 & 	 99.6890 	 &  	110.4082 	 &  	121.1274 	 &  	131.8467 	 &  	142.5659 	  \\
& 	 2 	 & 	 99.3798 	 &  	110.0658 	 &  	120.7519 	 &  	131.4379 	 &  	142.1239 	  \\
& 	 3 	 & 	 99.0726 	 &  	109.7256 	 &  	120.3786 	 &  	131.0316 	 &  	141.6845 	  \\
& 	 4 	 & 	 98.7673 	 &  	109.3875 	 &  	120.0076 	 &  	130.6278 	 &  	141.2479 	  \\
\hline\parbox[t]{2mm}{\multirow{6}{*}{\rotatebox[origin=c]{90}{\textrm{\scriptsize akC}}}}& \backslashbox{$t$}{$k$} &0 	 & 	1 	 & 	2 	 & 	3 	 & 	4\\ \cline{2-7}
& 	 0 	 & 	 100 & 	 & 	 & 	 & 	\\
& 	 1 	 & 	 98.8537 	 &  	117.2536 	 &  	135.6534 	 &  	154.0533 	 &  	172.4531 	  \\
& 	 2 	 & 	 97.7391 	 &  	115.9046 	 &  	134.0700 	 &  	152.2355 	 &  	170.4009 	  \\
& 	 3 	 & 	 96.6548 	 &  	114.5925 	 &  	132.5301 	 &  	150.4677 	 &  	168.4053 	  \\
& 	 4 	 & 	 95.5996 	 &  	113.3157 	 &  	131.0317 	 &  	148.7478 	 &  	166.4639 	  \\
\hline\parbox[t]{2mm}{\multirow{6}{*}{\rotatebox[origin=c]{90}{\textrm{\scriptsize HURWICZ$_{0.5}$}}}}& \backslashbox{$t$}{$k$} &0 	 & 	1 	 & 	2 	 & 	 3 	 & 	4\\ \cline{2-7}
& 	 0 	 & 	 100 & 	 & 	 & 	 & 	\\
& 	 1 	 & 	 55.5666 	 &  	71.8029 	 &  	88.0392 	 &  	104.2754 	 &  	120.5117 	  \\
& 	 2 	 & 	 39.6633 	 &  	50.9655 	 &  	62.2677 	 &  	73.5699 	 &  	84.8721 	  \\
& 	 3 	 & 	 31.4239 	 &  	40.1849 	 &  	48.9458 	 &  	57.7067 	 &  	66.4676 	  \\
& 	 4 	 & 	 26.3470 	 &  	33.5529 	 &  	40.7588 	 &  	47.9647 	 &  	55.1706 	  \\
\hline\parbox[t]{2mm}{\multirow{6}{*}{\rotatebox[origin=c]{90}{\textrm{\scriptsize HURWICZ$_{0.7}$}}}}& \backslashbox{$t$}{$k$} &0 	 & 	1 	 & 	2 	 & 	 3 	 & 	4\\ \cline{2-7}
& 	 0 	 & 	 100 & 	 & 	 & 	 & 	\\
& 	 1 	 & 	 59.1645 	 &  	75.2508 	 &  	91.3371 	 &  	107.4234 	 &  	123.5097 	  \\
& 	 2 	 & 	 44.3836 	 &  	55.9269 	 &  	67.4702 	 &  	79.0135 	 &  	90.5568 	  \\
& 	 3 	 & 	 36.6131 	 &  	45.7996 	 &  	54.9860 	 &  	64.1724 	 &  	73.3589 	  \\
& 	 4 	 & 	 31.7437 	 &  	39.4756 	 &  	47.2074 	 &  	54.9393 	 &  	62.6712 	  \\
\hline
\end{tabular}
\caption{Geometric--Beta Bonus Malus premiums by using OWA operators.\label{GB}}
\end{table}

\section{Conclusions}
\label{sec:4}

We present a new family of flexible premium computation principles that allows to incorporate the beliefs of several experts in computing the collective and Bayes premiums. This framework allows to model different scenarios in calculating premiums: equitable, pessimistic, optimistic, trimmed and some mixtures of pessimistic--optimistic cases. We present mathematical programming formulations to compute those premiums as well as a finite search algorithm that allows to compute efficiently the premiums for any choice of the elements in the new family of premiums computation principles.

We prove some properties of the premiums computed under this new framework and we run some applications with real data sets in order to illustrate the methodology and to compare the results with those obtained with the classical approach.

There are many possible extensions of this work. The use of severities of the claims, and then, dealing with the compound model is the natural first step. Also, the use of measures that allow to choose among the many possible choices of $\omega$-weights would be a nice topic for further research. For the sake of that, one may first make the premiums comparable when computed with different approaches. Hence, budget-like constraints should be added to the mathematical programming problem that models the computation principle. Our finite search algorithm is not capable to handle such a variation, being necessary the development of a different solution approach. Furthermore, note that it  is not trivial to implement a similar approach to other types of premium computation principles (variance, exponential, Esscher, etc) since the loss functions involved are not convex or quadratic anymore. However, the use of (linear or convex quadratic) approximations of those functions may help to adapt the OWA methodology to other cases.

\section*{Acknowledgements}

The first author was partially supported by the projects FQM-5849 (Junta de Andaluc\'ia $\backslash$ FEDER), MTM2013-46962-C2-1-P (MINECO, Spain) and MTM2016-74983-C2-1-R. The second author was partially supported by the project ECO2013-4702 (MINECO, Spain). The authors were also supported by the research group SEJ-534 (Junta de Andaluc\'ia).

\bibliographystyle{elsarticle-harv}

\appendix

\section{Proof of Theorem \ref{theo:0}}\label{appendix}

In this appendix we give the details about the correctness and complexity of Theorem \ref{theo:0}. It is based on the following key points:
\begin{itemize}
\item\label{item1} The set of points where the loss functions of each of the experts change their sorting is finite. Hence, if we restrict optimizing  $\mathcal{L}^\ell_{OWA}$ between two of those ``breaking points'' the function is a nonnegative linear combination of convex functions, then so is the function for $\ell \in \{C,B\}$. Thus, both $\mathcal{L}^B$ and $\mathcal{L}^C$ are piecewise convex functions.
\item For each domain $I_j$, because of item \ref{item1} $\mathcal{L}^\ell_{OWA}(P)$, $\mathcal{L}^\ell_{OWA}(P)$ is convex in $I_j$. Hence, $\mathcal{L}^\ell_{OWA}(P)$ has at most one critical point in $I_j$, for all $j = 1, \ldots, k$, $\ell \in \{C, B\}$.
\item  In each of the intervals $I_j$, the losses for each of the experts are completely ordered. Hence, a single evaluation of a point inside the interval allows us to sort the losses of the experts and compute the ordered weighted sum $\mathcal{L}^\ell_{OWA}$.
\item Since $L^\ell_{OWA}(P)$ is strictly convex in $I_j$, those optimal solutions can be obtained by applying the standard optimality conditions over the closure of $I_j$, for $j=1, \ldots, n$.
\end{itemize}

The above comments allows to assure that $\mathcal{A}$ always contains the OWA premium. Since the number of elements in $\mathcal{A}$ is at most $n^2$, the overall  worst case complexity is $O(n^2)$.

\end{document}